\documentclass[11pt]{article}
\usepackage[pdftex]{graphicx}
\usepackage[inline]{enumitem}
\usepackage{mathtools}
\usepackage{amssymb}
\usepackage{amsthm}
\usepackage{dsfont}
\usepackage{chngcntr}
\usepackage{multirow}
\usepackage{graphicx}
\usepackage{caption}
\usepackage{subcaption}
\usepackage{longtable}
\usepackage{amsmath}

\renewenvironment{proof}[1][\proofname]{{\bfseries #1.}}{\qed}
\newtheorem{thm}{Theorem}[section]
\newtheorem{lemma}{Lemma}[section]
\theoremstyle{remark}
\newtheorem{ass}{Assumption}
\theoremstyle{definition}
\newtheorem{remark}{Remark}[section]

\newcounter{step}
\newenvironment{steps}
{\begin{list}{\it{Step\,\arabic{step}} :}
{\usecounter{step}}}
{\end{list}}

\newcommand*{\rom}[1]{\expandafter{\romannumeral #1\relax}}

\newcommand{\E}{\mathbb{E}}
\newcommand{\Var}{\mathrm{Var}}

\newcommand{\C}{\mathcal{C}}
\allowdisplaybreaks

\renewcommand{\baselinestretch} {1.2}

\makeatletter 
\def\singlespace{\def\baselinestretch{1}\@normalsize}

\title{{\sc Testing  equality of autocovariance operators for functional time series}}

\author{ {Dimitrios ~{\sc PILAVAKIS}\footnote{
          Email: \texttt{pilavakis.dimitrios@ucy.ac.cy}}}, \; {Efstathios ~{\sc PAPARODITIS}\footnote{
          Email: \texttt{stathisp@ucy.ac.cy}}} \; and \;  {Theofanis ~{\sc SAPATINAS}\footnote{Email: \texttt{fanis@ucy.ac.cy}~ (Corresponding author)}} \\
          Department of Mathematics and Statistics, University of Cyprus, \\
          P.O. Box 20537, CY 1678 Nicosia, CYPRUS.
}

\date{}
\begin{document}
\maketitle

\begin{abstract}
We consider strictly stationary stochastic processes of Hilbert space-valued random variables and focus on 
 fully functional tests for the equality of the lag-zero autocovariance operators of several independent functional time series.
A  moving block bootstrap-based testing procedure is proposed  which  generates pseudo random elements  that satisfy the null hypothesis of interest.  It is  based on directly bootstrapping  the time series of tensor products  which overcomes  
some  common difficulties associated  with applications of the bootstrap to related  testing problems. The suggested methodology  can  be potentially applied to a broad range of test statistics of the hypotheses of interest. As an example, we establish validity for  approximating the distribution under the null of  a  test statistic
based on the Hilbert-Schmidt distance of the corresponding sample lag-zero autocovariance operators, and show consistency under the alternative. As a prerequisite, we prove a central limit theorem for the moving block bootstrap procedure applied to the sample autocovariance operator which is of interest on its own.  The finite sample size and power performance of the suggested moving block bootstrap-based testing procedure is illustrated through simulations  and an application to a real-life dataset is discussed.

\medskip
\noindent
{\em Some key words:} {\sc
 Autocovariance Operator;
 Functional Time Series;
 Hypothesis Testing;
 Moving Block Bootstrap.}

\end{abstract}

\section{\sc Introduction}
Functional data analysis deals with random variables which are curves or images and can be expressed as functions in appropriate spaces. In this paper, we consider functional time series $\mathds{X}_n=\{ X_1,X_2,\ldots,X_n\}$ steming from a strictly stationary stochastic process $\mathds{X}=(X_t,\,t\in \mathbb{Z})$ of Hilbert space-valued random functions  $ X_t(\tau),\,\tau\in \mathcal{I}$ (where $\mathcal{I}$ is a compact interval on $\mathbb{R}$), which are assumed to be  $L^4$-$m$-approximable, a dependence assumption which is satisfied by large classes of  commonly used functional time series models; see, e.g., H\"{o}rmann and  Kokoszka (2010).  We would like to infer properties of a group of  $K$ independent functional  processes 
based on  observed stretches  from each group. In particular, we focus on  the problem of testing whether the lag-zero autocovariance operators of   the $K$  processes are equal  and consider  fully functional  test statistics which evaluate the difference between the corresponding sample lag-zero autocovariance operators using appropriate distance measures.

\medskip

As it is common in the statistical analysis of functional data,  the limiting distribution of such statistics depends, in a complicated way, on difficult to estimate characteristics of the underlying functional stochastic  processes like,  for instance, its entire fourth order temporal dependence structure.  
Therefore, and in order to implement the  testing approach proposed,
 we apply   a moving block bootstrap (MBB) procedure which is used to estimate the distribution of the test statistic of interest under the null. 
 Notice that for testing problems related to the  equality of second order characteristics of several independent groups, 
  in the finite or infinite dimensional setting, applications of the bootstrap to approximate the distribution of  a test statistic of interest are  commonly based on  the generation of  pseudo random observations obtained by  resampling from the pooled (mixed) sample consisting of all available observations. Such implementations  lead to  the problem that the generated 
  pseudo observations  have not only identical second order characteristics but also identical distributions. This may affect the power and the  conditions needed for bootstrap consistency in that it may restrict its validity to specific situations only;
  see  Lele and Carlstein (1990) for an overview for the case of independent and identically distributed (i.i.d.) real-valued random variables   and Remark~\ref{rem:MBB} in Section 3 below for more details in the functional setting. 
  
  \medskip 
  
  To overcome such problems, we use a different  approach which is based on the observation that the 
  lag-zero autocovariance operator $\mathcal{C}_0=\E(X_t-\mu)\otimes (X_{t}-\mu)$ is the expected value of the tensor product process $ \{{\mathcal Y}_t=(X_t-\mu)\otimes(X_{t}-\mu)$, $t\in \mathbb{Z}$\},  where $\mu=\E X_t$ denotes the expectation of $X_t$.  Therefore, the  testing problem of interest can also be viewed as testing for the equality of expected values (mean functions) of the associated  processes of tensor products.  The  suggested MBB procedure works by first generating  functional pseudo random elements  via resampling from the time series of tensor products of 
  the same group and then 
  adjusting  the mean function of the generated pseudo random elements in  each group so that the null hypothesis of interest is satisfied.
We stress here the fact  that the proposed method is not designed having  any particular test statistic in mind and it is, therefore, potentially  applicable to a wide range of different test statistics.  As an example,  we establish in this paper validity of the proposed MBB-based testing procedure in estimating  the distribution of a particular fully functional test statistic under the null,  which is based on the Hilbert-Schmidt norm between the sample lag-zero autocovariance operators, and show its consistency under the alternative. By fully functionals tests, we mean tests which exploit the entire infinite dimensionality structure of the underlying stochastic process and do not attempt to reduce dimensionality by projecting on finite dimensional subspaces. The idea of block bootstrapping from blocks is not new and have been previously investigated by  K\"{u}nsch (1989) for a fixed number of blocks and by  Politis and Romano (1992)  in a more general context where  the number of blocks is allowed to increase to infinity with the sample size $n$.   Furthermore, by considering the aforementioned tensor products, the problem of testing for differences in the autocovariance operators becomes similar to the functional ANOVA problem; see  Cuevas {\em et al.} (2004),  Zhang (2013),  Horv\'ath and Rice (2015) and H\"ormann {\em et al.} (2018). 

\medskip

As a prerequisite, to our theoretical derivations, we first prove a central limit theorem for the MBB procedure applied to the sample version of the autocovariance operator $\mathcal{C}_h=\E(X_t-\mu)\otimes (X_{t+h}-\mu)$, $h \in \mathbb{Z}$, of an $L^4$-$m$-approximable stochastic process, which is of interest on its own. Our  results imply  that the suggested MBB-based testing procedure is not restricted  to the case of testing for equality of  the lag-zero autocovariance operator only but it can be adapted to tests dealing  with the equality of  any (finite number of)  autocovariance operators $ \mathcal{C}_h$  for  lags  $h $  different from zero. 

\medskip
Asymptotic and bootstrap based inference procedures for covariance operators for two or more populations of i.i.d.  functional data have been extensively discussed in the literature; see, e.g., Panaretos {\em et al.} (2010), Fremdt {\em et al.} (2013)  for tests based on finite-dimensional projections, Pigoli {\em et al.} (2014)  for permutation tests  based on distance measures and Paparoditis and Sapatinas (2016) for fully  functional tests. 
Notice that testing for the equality of the lag-zero autocovariance operators  is an important problem for functional time series since  the  associated covariance kernel 
$ c_0(u,v)= Cov(X_t(u), X_t(v)) $ of the lag-zero autocovariance operator  $\mathcal{C}_0$ describes,  for $ (u,v) \in \mathcal{I}\times \mathcal{I}$, the entire covariance structure  of the   random function  $ X_t$.
  Despite its importance,  this testing 
  problem has been considered, to the best of our knowledge, only recently by Zhang and Shao (2015). 
 To tackle the aforementioned problems associated  with the  implementability of  limiting distributions, Zhang and Shao (2015) considered tests based on  projections on finite dimensional spaces of the differences of the estimated  lag-zero autocovariance operators. Notice that similar directional tests have previously been considered for i.i.d. functional data; see   Panaretos {\em et. al.} (2010) and Fremdt {\em et al.} (2013). Although projection-based tests  have the advantage that they  lead to manageable limiting distributions, and can be powerful when  the deviations from the null are captured  by the finite-dimensional space projected, such tests  have no power  for  alternatives which are orthogonal to the projection space.  Moreover, and apart from being free from the choice of testing parameters, like the choice of the dimension of the projection space, and from being consistent for a broader class of alternatives, the fully functional tests considered in this paper also allow for a nice interpretation of the test results obtained; we refer to Section 4 for an example.  

\medskip
The paper is organised as follows. In Section~\ref{sec:notations}, the basic assumptions on the underlying stochastic process $\mathds{X}$ are stated and the asymptotic validity of the MBB procedure applied to  estimate the distribution of the sample autocovariance operator is established. In Section~\ref{sec:algorithm}, the proposed MBB-based procedure for testing equality of the lag-zero autocovariance operators for several independent functional time series is introduced. Theoretical justifications for approximating the null distribution of a particular fully functional test statistic are given and consistency under the alternative is obtained. Numerical simulations are presented in Section~\ref{sec:numerical} in which the finite sample behaviour of the proposed MBB-based testing methodology is investigated. A Cyprus daily temperature data example is also discussed in this section. Auxiliary results and proofs of the main results are deferred to Section~\ref{sec:proofs} and to the supplementary material.

\section{\sc Bootstrapping the autocovariance operator}
\label{sec:notations}
\subsection{\sc Preliminaries and Assumptions}
We consider a strictly stationary stochastic process $\mathds{X}=\{X_t,\,t\in \mathbb{Z}\},$ where the random variables $X_t$ are random functions $ X_t(\omega, \tau),\,\tau\in \mathcal{I},\,\omega\in\Omega,\,t\in\mathbb{Z}$, defined on a probability space $(\Omega,A, P)$ and take values in the separable Hilbert space of squared-integrable $\mathbb{R}$-valued functions on $\mathcal{I}$, denoted by $L^2(\mathcal{I}).$ The expectation function of $X_t$, $\E X_t\in L^2(\mathcal{I}),$ is independent of $t,$ and it is denoted by $\mu.$
We define $\langle f, g\rangle =\int_{\mathcal{I}}f(\tau)g(\tau)\mathrm{d}\tau,$ $\|f\|^2=\langle f,f\rangle$ and the tensor product between $f$ and $g$ by $f\otimes g(\cdot)=\langle f, \cdot\rangle g.$ For two Hilbert Schmidt operators $\Psi_1$ and $\Psi_2,$ we denote by $\langle \Psi_1,\Psi_2\rangle_{HS} =\sum_{i=1}^{\infty}\langle\Psi_1(e_i),\Psi_2(e_i)\rangle$ the inner product which generates the Hilbert Schmidt norm $\|\Psi_1\|_{HS}=\sum_{i=1}^{\infty}\|\Psi_1(e_i)\|^2$, where $\{e_i,i=1,2,\ldots\}$ is any orthonormal basis of $L^2(\mathcal{I}).$ If $\Psi_1$ and $\Psi_2$ are Hilbert Schmidt integral operators with kernels $\psi_1(u,v)$ and $\psi_2(u,v),$ respectively, then $\langle \Psi_1,\Psi_2\rangle_{HS}=\int_{\mathcal{I}}\int_{\mathcal{I}}\psi_1(u,v)\psi_2(u,v)\mathrm{d}u\mathrm{d}v.$ We also define the tensor product between the operators $\Psi_1$ and $\Psi_2$ analogous to the tensor product of two functions, i.e., $\Psi_1\otimes \Psi_2(\cdot)=\langle \Psi_1, \cdot\rangle_{HS} \Psi_2.$ Note that $\Psi_1\otimes \Psi_2$ is an operator acting on the space of Hilbert Schmidt operators.
Without loss of generality, we assume that $\mathcal{I} = [0, 1]$ (the unit interval) and, for simplicity, integral signs without the limits of integration imply integration over the interval $\mathcal{I}.$ We finally write $L^2$ instead of $L^2(\mathcal{I})$, for simplicity. For more details, we refer to Horv\'{a}th and Kokoszka (2012, Chapter 2).

\medskip
To describe more precisely the dependence  structure of the stochastic process $\mathds{X}$, we use the notion of $L^p$-$m$-approximability; see H\"{o}rmann and  Kokoszka (2010). A stochastic process $\mathds{X}=\{X_t,t\in\mathbb{Z}\}$ with $X_t$ taking values in $L^2$, is called $L^4$-$m$-approximable if the following conditions are satisfied:
\begin{enumerate}[label=(\roman*)]
     \item $X_t$ admits the representation \begin{equation}\label{orismosL2a}X_t=f(\delta_t,\delta_{t-1},\delta_{t-2},\ldots)\end{equation}  for some measurable function $f:S^\infty\rightarrow L^2$, where $\{\delta_t,\,t\in\mathbb{Z}\}$ is a sequence of i.i.d. elements in $L^2$.
\item $\E\|X_0\|^4<\infty$ and
    \begin{equation}\label{orismosL2b}
    \sum_{m\geq1}\left(\E\|X_t-X_{t,m}\|^4\right)^{1/4}<\infty,\end{equation} where $X_{t,m}=f(\delta_t,\delta_{t-1},\ldots,\delta_{t-m+1},\delta_{t,t-m}^{(m)},\delta_{t,t-m-1}^{(m)},\ldots)$
    and, for each $t$ and $k,$ $\delta_{t,k}^{(m)}$ is an independent copy of $\delta_t.$
\end{enumerate}
The rational behind this concept of weak dependence is that the 
 function $f$ in~\eqref{orismosL2a}  is such   that the effect of the innovations $\delta_i$ far back in the past becomes negligible, that is, these innovations can be replaced by other, independent, innovations. For the stochastic process $\mathds{X}$ considered in this paper, we somehow strengthen \eqref{orismosL2b} to the following assumption.
\begin{ass} \label{as:strongL4}
$\mathds{X}$ is $L^4$-$m$-approximable and satisfies
$$\lim_{m\to\infty}m\left({\E\|X_t-X_{t,m}\|^4}\right)^{1/4}=0.$$
\end{ass}

\medskip
\noindent
Since $\E\|X_t\|^2<\infty$,  the autocovariance operator at lag $h \in \mathbb{Z}$ exists and is defined by 
$$
\mathcal{C}_h=\E[(X_t-\mu)\otimes (X_{t+h}-\mu)].
$$
Having an observed stretch $X_1,X_2,\ldots,X_n$, the operator $ \mathcal{C}_h$  is 
commonly estimated by the corresponding sample autocovariance operator, which is given by

\begin{equation*}
 \mathcal{\widehat{C}}_{h}=
 \begin{cases}
      n^{-1}\sum_{t=1}^{n-h}( X_t-\overline{X}_n)\otimes (X_{t+h}-\overline{X}_n), & \text{if}\ 0\leq h<n, \\
      n^{-1}\sum_{t=1}^{n+h}( X_{t-h}-\overline{X}_n)\otimes (X_{t}-\overline{X}_n), & \text{if}\ -n<h<0 ,\\
      0, & \text{otherwise,}
    \end{cases}
  \end{equation*}
where $\overline{X}_n=(1/n)\sum_{t=1}^{n}X_t$ is the sample mean function.  The limiting distribution of $ \sqrt{n}\big(\mathcal{\widehat{C}}_{h} - \mathcal{C}_h\big)$  can be derived using the same arguments to those applied in Kokoszka na Reimherr (2013) to investigate the limiting distribution of $ \sqrt{n}\big(\mathcal{\widehat{C}}_{0} - \mathcal{C}_0\big)$.
More precisely, it can  be shown that, for any (fixed) lag $h,$ $h \in \mathbb{Z}$, under  $L^4$-approximability conditions,
  $  \sqrt{n}\big(\mathcal{\widehat{C}}_{h} - \mathcal{C}_h\big) \Rightarrow  \mathcal{Z}_h$, where $\mathcal{Z}_h$ is a Gaussian Hilbert-Schmidt operator   
  with covariance operator $\Gamma_h$ given by  
  $$
\Gamma_h =
\sum_{s=-\infty}^{\infty}\E[((X_{1}-\mu)\otimes (X_{1+h}-\mu)-\mathcal{C}_h)\otimes ((X_{1+s}-\mu)\otimes (X_{1+h+s}-\mu)-\mathcal{C}_{h})];
$$
see also 
Mas (2002) for a related   result if  $\mathds{X}$ is a Hilbertian linear processes.

\subsection{\sc A Bootstrap CLT  for the empirical autocovariance operator}

In this section, we formulate and prove consistency of the MBB  for estimating the distribution of   
$\sqrt{n}\big(\mathcal{\widehat{C}}_{h} - \mathcal{C}_h\big) $  for any (fixed) lag $h,$ $h \in \mathbb{Z}$, in the case of weakly dependent Hilbert space-valued random variables satisfying the $ L^4$-approximability condition stated in Assumption 1. The MBB procedure 
was originally proposed for real-valued time series by K\"{u}nsch $(1989)$ and Liu and Singh $(1992)$.  
Adopted to the functional set-up,  this resampling procedure 
 first divides the functional time series at hand   into the  collection of all possible overlapping blocks of functions of length $b$. That is, 
  the first block consists of the functional observations 1 to $b$, the second block consists of the functional observations 2 to $b+1$, and so on. Then,  a bootstrap sample is  obtained by independent sampling, with replacement, from these blocks of functions and joining the blocks together in the order selected to form a new set of functional pseudo observations.
  
 However, to deal with  the problem of  estimating  the distribution of the sample autocovariance operator  $\mathcal{\widehat{C}}_{h}$,  we modify the above basic idea and apply the MBB directly to the set of random elements $\mathds{Y}_{n-h} =\{\mathcal{\widehat{Y}}_{t,h}, \, t=1,2,\ldots,n-h \}$, where $\mathcal{\widehat{Y}}_{t,h}=(X_{t}-\overline{X}_n)\otimes(X_{t+h}-\overline{X}_n)$.
  As mentioned in the Introduction, this has certain advantages in the testing context which will be discussed  in the next section. 
The  MBB procedure  applied to generate  the pseudo random elements 
$\mathcal{Y}_{1,h}^*,\mathcal{Y}_{2,h}^*,\ldots,\mathcal{Y}_{n-h,h}^*$
 is described by  the following steps.
\begin{steps}
  \item Let $b=b(n),1\leq b<n-h$, be an integer and denote by $B_t = \{\mathcal{\widehat{Y}}_{t,h},\mathcal{\widehat{Y}}_{t+1,h},\ldots,\mathcal{\widehat{Y}}_{t+b-1,h}\}$ the block of length $b$ starting from the tensor operator $\mathcal{\widehat{Y}}_t,$ where $ t=1,2,\ldots,N$  and  $N=n-h-b+1$ is the total number of such blocks available.
  \item Let $k$ be a positive integer satisfying $b(k-1)<n-h$ and $bk\geq n-h$ and define $k$ i.i.d. integer-valued random variables $I_1, I_2,\ldots, I_k$ selected from a discrete uniform distribution which assigns probability $1/N$ to each element of the set $\{1,2,\ldots,N\}$.
  \item Let $B_i^*=B_{I_i},\,i=1,2,\ldots,k$, and denote by
$\{\mathcal{Y}^*_{(i-1)b+1,h},\mathcal{Y}^*_{(i-1)b+2,h},\ldots,\mathcal{Y}^*_{ib,h}\}$ the elements of $B_i^*.$ Join the $k$ blocks in the order $B_1^*,B_2^*,\ldots,B_k^*$  together to obtain a new set of functional pseudo observations. The MBB generated sample of pseudo 
random elements  consists then of the set  $\mathcal{Y}_{1,h}^*,\mathcal{Y}_{2,h}^*,\ldots,\mathcal{Y}_{n-h,h}^*.$
\end{steps}
Note  that if we are  interested in the distribution of the sample autocovariance operator $\mathcal{\widehat{C}}_{h}$ for  some (fixed) lag $h$, $-n<h<0$, then the above algorithm can be applied to the time series of operators $\mathds{Y}_{n+h}=\{\mathcal{\widehat{Y}}_{t,h},\,t=h+1,h+2,\ldots,n\}$, where $\mathcal{\widehat{Y}}_{t,h}=(X_{t-h}-\overline{X}_n)\otimes(X_{t}-\overline{X}_n),\,t=h+1,h+2,\ldots,n,$ with minor changes. Hence, below, we only focus on the case of $0\leq h < n$.

Given a stretch $\mathcal{Y}_{1,h}^*,\mathcal{Y}_{2,h}^*,\ldots,\mathcal{Y}_{n-h,h}^*$   of pseudo random elements 
 generated by the above MBB procedure,  a bootstrap estimator of the autocovariance operator is given by the sample mean 
$$
\mathcal{\widehat{C}}_{h}^*=\frac{1}{n}\sum_{t=1}^{n-h}\mathcal{Y}^*_{t,h}.
$$
The  proposal is then  to estimate the distribution of $\sqrt{n}(\mathcal{\widehat{C}}_{h}-\mathcal{C}_h)$ by the distribution of the  bootstrap analogue
 $\sqrt{n}(\mathcal{\widehat{C}}_{h}^*-\E^*(\mathcal{\widehat{C}}_{h}^*))$, where $\E^*(\mathcal{\widehat{C}}_{h}^*)$ is (conditionally on $\mathds{X}_n$) the expected value of $\mathcal{\widehat{C}}_{h}^*.$  Assuming, for simplicity, that $n-h=kb,$ straightforward calculations yield
\begin{equation} \label{eq.mean-MBB-star}
\E^*(\mathcal{\widehat{C}}_{h}^*)=\dfrac{1}{N}\dfrac{n-h}{n}\left[\sum_{t=1}^{n-h}\mathcal{\widehat Y}_{t,h}-\sum_{j=1}^{b-1}\Big(1-\frac{j}{b}\Big)(\mathcal{\widehat Y}_{j,h}+\mathcal{\widehat Y}_{n-h-j+1,h})\right].
\end{equation}

The following theorem establishes validity of the MBB procedure suggested 
for approximating the distribution of  $\sqrt{n}(\hat{C}_{h}-C_h)$.

\begin{thm}
\label{thm:CLTmbb}
Suppose that the stochastic process $\mathds{X}$ satisfies Assumption~\ref{as:strongL4}. For $0\leq h<n,$ let $\mathcal{Y}_{1,h}^*,\mathcal{Y}_{2,h}^*,\ldots,\mathcal{Y}_{n-h,h}^*$ be a stretch of functional 
pseudo random elements  generated as in  Steps 1-3 of  the MBB procedure and assume that the block size $b=b(n)$ satisfies $b^{-1}+bn^{-1/3}=o(1)$ as $n\to\infty.$ Then,
as $n \to \infty,$
$$
 d(\mathcal{L}(\sqrt{n}(\mathcal{\widehat C}_{h}^*-\E^*(\mathcal{\widehat C}_{h}^*)) \mid\mathds{X}_n),\;\mathcal{L}(\sqrt{n}(\mathcal{\widehat C}_{h}-\mathcal{C}_h)))\to 0,
\quad\text{in probability,}
$$
where $d$ is any metric metrizing weak convergence on the space of Hilbert-Schmidt operators acting on $L^2$ and $\mathcal{L}(Z)$ denotes the law of the random element $Z$ belonging to this operator space.
\end{thm}

\section{\sc Testing equality of lag-zero autocovariance operators}
\label{sec:algorithm}
In this section, we consider 
the problem of testing  the equality of the lag-zero autocovariance operators for  a finite number of functional time series and use a modified version of the propopsed MBB procedure. This modification leads to a MBB-based testing procedure which generates functional pseudo observations that satisfy the null hypothesis that all lag-zero autocovariance operators are equal. Since this procedure is designed without having any particular statistic in mind, it 
can  potentially be applied to a broad range of possible test statistics which are appropriate  for the particular testing problem considered. 


To make things specific, consider $K$ independent,  $L^4$-$m$-approximable functional time series, denoted in the following by  $\mathds{X}_{K,M}=\{X_{i,t},\,i=1,2\ldots,K,\,t=1,2,\ldots,n_i\},$ where $K$ denotes the number of time series and $M = \sum_{i=1}^Kn_i$ the total number of observations, with $n_i$ denoting the length of the $i$-th time series. Let ${\mathcal C}_{i,0},\,i=1,2\ldots,K,$ be the lag-zero autocovariance operator of the $i$-th functional time series, i.e., $\mathcal{C}_{i,0}=\E[(X_{i,t}-\mu_i)\otimes(X_{i,t}-\mu_i)]$, where $\mu_i   = EX_{i,t}$.
The null hypothesis of interest is then
\begin{equation}
\label{eq-cov-fanis}
H_0:{\mathcal C}_{1,0}={\mathcal C}_{2,0}=\ldots={\mathcal C}_{K,0} 
\end{equation} 
and the alternative hypothesis is
$$
 H_1:\exists\, k,m \in\{1,2,\ldots,K\}\; \text{with}\; k\neq m\;\; \text{such that}\;\;  \|{\mathcal C}_{k,0}- {\mathcal C}_{m,0}\|_{HS}>0. 
 $$
By considering the operator processes  $\{\mathcal{Y}_{i,t}=(X_{i,t}-\mu_i)\otimes(X_{i,t}-\mu_i), t \in \mathbb{Z}\}$, \, $i=1,2\ldots,K$,  and denoting by  $\mu^\mathcal{Y}_i=\E \mathcal{Y}_{i,t}$  the expectation of   $\mathcal{Y}_{i,t}$, the  null hypothesis of interest can be equivalently written as
\begin{equation}
\label{eq-cov-tensor-fanis}
H_0:\mu^\mathcal{Y}_1=\mu^\mathcal{Y}_2=\ldots=\mu^\mathcal{Y}_K 
\end{equation} 
and the alternative hypothesis as
$$ H_1:\exists\, k,m \in\{1,2,\ldots,K\}\; \text{with}\; k\neq m\;\; \text{such that}\;\; \|\mu^\mathcal{Y}_k-\mu^\mathcal{Y}_m\|_{HS}>0.$$
Consequently, the aim  of the bootstrap is to generate a set of $K$  pseudo random elements $\mathds{Y}^*_{K,M}=\{\mathcal{Y}^*_{i,t},\,i=1,2\ldots,K$, $t=1,2,\ldots,n_i\}$ which satisfy the null hypothesis (\ref{eq-cov-tensor-fanis}), that is,  the  expectations  $E^*(\mathcal{Y}^*_{i,t})$ should be identical for all 
   $i=1,2,\ldots,K$.  This leads to  the  MBB-based testing procedure described  in the next section. \par

\subsection{\sc The MBB-based Testing Procedure}
\label{ssec:algorithm}

Suppose that, in order to test  the null hypothesis (\ref{eq-cov-tensor-fanis}), we use a   real-valued test statistic $T_M$, where, for simplicity, we assume that large values of $T_M$ argue against the null hypothesis. Since 
we focus on the tensor operators  $\mathcal{Y}_{i,t},\,t=1,2,\ldots,n_i,\,i=1,2\ldots,K,$ it is natural to assume that 
 the test statistic  $T_M$ is   based on the tensor product of the centered observed functions, that is on 
$\widehat{\mathcal{Y}}_{i,t}=(X_{i,t}-\overline{X}_{i,n_i})\otimes(X_{i,t}-\overline{X}_{i,n_i}),\,i=1,2\ldots,K,\,t=1,2,\ldots,n_i,$ where $\overline{X}_{i,n_i}$ is the sample mean function of the $i$-th population, i.e, $\overline{X}_{i,n_i}=(1/n_i)\sum_{t=1}^{n_i}X_{i,t}$. Suppose next, without los of generality,  that the null hypothesis (\ref{eq-cov-tensor-fanis}) is rejected if $T_M > d_{M,\alpha},$ where, for $\alpha\in(0,1)$,  $d_{M,\alpha}$ denotes the upper $\alpha$-percentage point of the distribution of $T_M$ under $H_0$.  We propose to approximate the distribution of $T_M$ under $H_0$ by 
the distribution of the  bootstrap quantity $T^*_M$, where the latter is obtained through the following steps.
\begin{steps}
\item Calculate the pooled mean 
$$
\overline{\mathcal{Y}}_M=\dfrac{1}{M}\sum_{i=1}^{K}\sum_{t=1}^{n_i}\widehat{\mathcal{Y}}_{i,t}.
$$
\item For $i=1,2,\ldots,K$, let $b_i=b_i(n)\in\{1,2,\ldots,n-1\}$ be the block size used for the $i$-th functional time series and let $N_i=n_i-b_i+1$. Calculate
$$
\widetilde{\mathcal{Y}}_{i,\xi}=\dfrac{1}{N_i}\sum_{t=\xi}^{N_i+\xi-1}\widehat{\mathcal{Y}}_{i,t},\,\xi=1,2,\ldots,b_i
$$
\item For simplicity assume that $n_i=k_ib_i$ and for $i=1,2,\ldots,K$, let $q^i_1,q^i_2,\ldots, q^i_{k_i}$ be i.i.d. integers selected from a discrete probability distribution which assigns the probability $1/N_i$ to each element of the set $\{1,2,\ldots,N_i\}.$ Generate bootstrap functional pseudo observations $\mathcal{Y}^*_{i,t},\,t=1,2,\ldots,n_i,\,i=1,2,\ldots,K$, as
$$
\mathcal{Y}^*_{i,t}=\overline{\mathcal{Y}}_M+\widehat{\mathcal{Y}}^*_{i,t}-\widetilde{\mathcal{Y}}_{i,\xi},\,\,\xi=b_i\text{ if } t\bmod b_i=0\text{ and } \xi=t\bmod b_i\text{ otherwise,}
$$
where
$\widehat{\mathcal{Y}}^*_{i,\xi+(s-1)b_i}=\widehat{\mathcal{Y}}_{i,q_s^i+\xi-1},\,s=1,2\ldots,k_i$\, and $\xi=1,2,\ldots,b_i$
\item Let $T_M^*$ be the same statistic as $T_M$ but calculated using,  instead of  the $\widehat{\mathcal{Y}}_{i,t}$'s the bootstrap  pseudo random elements  $\mathcal{Y}^*_{i,t},\,t=1,2,\ldots,n_i$, $i=1,2,\ldots,K$.  Given  $\mathds{X}_{K,M}$, denote by $D^*_{M,T}$
the distribution of $T^*_M$.  Then for  $\alpha \in (0,1)$, the null hypothesis $H_0$ is rejected  if
$$T_M > d_{M,\alpha}^*,$$
where $d_{M,\alpha}^*$ denotes the upper $\alpha$-percentage point of the distribution of $T_M^*$, i.e., $\mathbb{P}(T_M^*>d_{M,\alpha}^*)=\alpha$.
\end{steps}
Notice that the distribution $D^*_{M,T}$ is unknown but  it can be evaluated by Monte-Carlo.\par

\medskip
Before establishing validity of the described  MBB procedure some remarks are in order.  Observe  that the mean $\widetilde{\mathcal{Y}}_{i,\xi}$ calculated in Step~2, is the (conditional on $\mathds{X}_{K,M}$) expected value of  $\widehat{\mathcal{Y}}^*_{i,q_s^i+\xi-1}$ for $\xi=b_i$ if $t\bmod b_i=0$ and $\xi=t\bmod b_i$ otherwise. This motivates the  definition  
$$
\mathcal{Y}^*_{i,t}=\overline{\mathcal{Y}}_M+\widehat{\mathcal{Y}}^*_{i,t}-\widetilde{\mathcal{Y}}_{i,\xi},\,t=1,2,\ldots,n_i,\,i=1,2,\ldots,K,
$$
used  in Step 3 of the MBB algorithm. This definition  ensures that the generated  pseudo random elements $\mathcal{Y}^*_{i,t},\,t=1,2,\ldots,n_i,\,i=1,2,\ldots,K$,
satisfy the null hypothesis (\ref{eq-cov-tensor-fanis}). In fact, it is easily seen that the pseudo random elements 
$ \mathcal{Y}^*_{i,t}$ 
have (conditional on $\mathds{X}_{K,M}$) an expected value which is equal to $\overline{\mathcal{Y}}_M$, that 
is $ E^*( \mathcal{Y}^*_{i,t}) =\overline{\mathcal{Y}}_M $ for all $ t=1,\ldots, n_i$ and $i=1,\ldots, K$. 


\subsection{\sc Validity of the MBB-based Testing Procedure}
Although the proposed MBB-based testing procedure is not designed having any specific test statistic in mind, 
establishing its validity requires the consideration of a  specific class of statistics. In the following, and for simplicity, we focus on the case of two independent population, i.e., $K=2$.  
In this case, a natural approach to test  equality of the lag-zero autocovariance operators  
is to consider a  fully functional test statistic which evaluates  the difference between the empirical lag-zero autocovariance operators, for instance,  to use the test statistic 
$$
T_M=\dfrac{n_1n_2}{M}\|\widehat{\mathcal C}_{1,0}-\widehat{\mathcal C}_{2,0}\|^2_{HS} = \dfrac{n_1n_2}{M}\| \mathcal{\overline{Y}}_{1,n_1}-\mathcal{\overline{Y}}_{2,n_2}\|^2_{HS},
$$
where $\mathcal{\overline{Y}}_{i,n_i}=(1/n_i)\sum_{t=1}^{n_i}\widehat{\mathcal{Y}}_{i,t}$, $i=1,2$, and $M=n_1+n_2$. The following lemma  delivers 
the asymptotic  distribution of  $T_M$ under $H_0$.
\begin{lemma}
\label{thm:katanomiTmH0}
Let $H_0$ hold true,  Assumption~\ref{as:strongL4} be satisfied and assume that, as $\min\{n_1,n_2\}\to\infty,$ $n_1/M\to\theta\in(0,1).$  Then,
$$
T_M\overset{d}\to\|\mathcal{Z}_0\|^2_{HS}
$$
where $\mathcal{Z}_0=\sqrt{1-\theta}\mathcal{Z}_{1,0}-\sqrt{\theta}\mathcal{Z}_{2,0}$ and $\mathcal{Z}_{i,0},\,i=1,2$, are two independent mean zero Gaussian Hilbert-Schmidt operators with covariance operators $\Gamma_{i,0} $, $i=1,2$, given by 
\begin{align*}
\Gamma_{i,0}&=\E[((X_{i,1}-\mu_i)\otimes (X_{i,1}-\mu_i)-\mathcal{C}_{i,0})\otimes ((X_{i,1}-\mu_i)\otimes (X_{i,1}-\mu_i)-\mathcal{C}_{i,0})]\\&\quad+2\sum_{s=2}^{\infty}\E[((X_{i,1}-\mu_i)\otimes (X_{i,1}-\mu_i)-\mathcal{C}_{i,0})\otimes ((X_{i,s}-\mu_i)\otimes (X_{i,s}-\mu_i)-\mathcal{C}_{i,0})].
\end{align*}
\end{lemma}

%


As it is seen from the above lemma, the limiting distribution of $T_M$  depends on the difficult to estimate  covariance operators $\Gamma_{i,0}$, $i=1,2,$ which describe   the entire  fourth order structure of the underlying functional processes  $\mathds{X}_i $. This makes the 
 implementation of the derived asymptotic result for calculating critical values of the $T_M$ test a difficult task. Theorem \ref{thm:consistentMBB} below shows that the MMB-based testing procedure 
estimates consistently the limiting distribution $ \|\mathcal{Z}_0\|^2_{HS}$ of the $T_M$ test and, consequently,  that it 
 can be applied  to estimate  the  critical values of interest.\\

For this, we  apply the MBB-based testing procedure introduced in Section~\ref{ssec:algorithm}  to generate $ \{\mathcal{Y}_{i,t}^*, t=1,2, \ldots. n_i\}$, $ i\in \{1,2\}$,  and use  the bootstrap pseudo statistic 
$$
T^*_M=\dfrac{n_1n_2}{M}\|\mathcal{\overline{Y}}^*_{1,n_1}-\mathcal{\overline{Y}}^*_{2,n_2}\|^2_{HS},
$$
where 
$\mathcal{\overline{Y}}_{i,n_i}^*=(1/n_i)\sum_{t=1}^{n_i}\mathcal{Y}_{i,t}^*$, $i=1,2$.
We then have the following result. 

\begin{thm}
\label{thm:consistentMBB}
Let Assumption~\ref{as:strongL4} be satisfied and assume that $\min\{n_1,n_2\}\to\infty,$ $n_1/M\to\theta\in(0,1)$. Also, for $i\in\{1,2\}$, let the block size $b_i=b_i(n)$ satisfies $b_i^{-1}+b_in_i^{-1/3}=o(1),$ as $n_i\to\infty$. Then,
$$
\sup_{x\in \mathbb{R}} \bigl| P(T_M^*\leq x \mid \mathds{X}_{K,M})-P_{H_0}(T_M\leq x) \bigr|\to 0,\,\,\,\,\text{in probability},
$$
where $P_{H_0}(X\leq \cdot)$ denotes the distribution function of the random variable $X$ when $H_0$ is true.
\end{thm}

\begin{remark}
If $H_1$ is true, that is if $\|\mathcal{C}_{1,0}-\mathcal{C}_{2,0}\|_{HS}=\|\E \mathcal{Y}_{1,t}-\E \mathcal{Y}_{2,t}\|_{HS}>0$, then  it is easily seen that $ T_M \rightarrow \infty$ under the conditions on $ n_1$ and $n_2$ stated  in Lemma~\ref{thm:katanomiTmH0}. This, together with 
 Theorem ~\ref{thm:consistentMBB}  and  Slutsky's theorem, imply consistency of the $T_M$ test  based on   bootstrap critical values obtained using the distribution of $T_M^*$, i.e., the  power of the test approaches unity, as $n_1,n_2\to\infty.$
\end{remark}

\begin{remark}
\label{rem:MBB} 
The advantage of our approach to  translate the testing problem considered  to a testing problem of equality of mean functions and to apply the bootstrap to the time series of tensor operators $\mathcal{Y}_{i,t}$, $t=1,2, \ldots, n_i$,  $i=1, \ldots, K$,    is manifested  in the generality under which validity of the MBB-based testing  procedure is established in Theorem~\ref{thm:consistentMBB}. To elaborate, a MBB approach which would select blocks from the pooled (mixed)  set of functional time series in order to generate  bootstrap  pseudo elements  which satisfy the null hypothesis,  will lead to the generation of $K$ new functional pseudo time series,  which asymptotically will imitate correctly the pooled second {\it and}~ the  fourth order  moment  structure of the underlying functional processes.  As a consequence,  the limiting distribution of $ T_M$ as stated in Lemma~\ref{thm:katanomiTmH0} and that of the  corresponding MBB analogue will coincide only if $ \Gamma_1=\Gamma_2$. This obviously restricts  the class of processes for which the 
MBB procedure is   consistent. In the more simple i.i.d. case, a  similar limitation exists  by the condition $\mathcal{B}_1=\mathcal{B}_2$ imposed in Theorem 1 of Paparoditis and Sapatinas (2016). 
Notice that this  limitation can be resolved by applying also in the i.i.d. case the basic bootstrap idea  proposed in this paper. 
That is, to first translate  the testing problem to one  of testing equality 
of means of  samples consisting of  the   i.i.d. tensor operators  and then to apply an appropriate  i.i.d. bootstrap procedure.
\end{remark}

\section{\sc Numerical Results}
\label{sec:numerical}
In this  section, we investigate via simulations the size and power  behavior of the 
MBB-based testing procedure applied to testing the equality of lag zero autocovariance operators and we illustrate its applicability 
 by considering   a real life data-set.
\subsection{\sc Simulations}
In the  simulation experiment, two functional time series  $X_{1,t}$ and $X_{2,t}$ are  generated  from  the  functional autoregressive (FAR) models, 
\begin{align}
\label{model:FAR1}
X_{1,t}(u)&=\int\psi(u,v)X_{1,t-1}(v)\,\mathrm{d}v+\delta X_{1,t-2}(u)+B_{1,t}(u)  \nonumber  \\ 
X_{2,t}(u)&=\int\psi(u,v)X_{2,t-1}(v)\,\mathrm{d}v+B_{2,t}(u) 
\end{align}
or from  the  functional moving average (FMA) models,
\begin{align}
\label{model:FMA1a}
X_{1,t}(u)&=\int\psi(u,v)B_{1,t-1}(v)\,\mathrm{d}v+\delta B_{1,t-2}(u)+B_{1,t}(u) \nonumber  \\
X_{2,t}(u)&=\int\psi(u,v)B_{2,t-1}(v)\,\mathrm{d}v+B_{2,t}(u).
\end{align}
The kernel function $\psi(\cdot,\cdot)$ in the above  models is equal and it is given  by
$$
\psi(u,v)=\frac{\displaystyle \mathrm{e}^{-(u^2+v^2)/2}}{\displaystyle 4\int\mathrm{e}^{-t^2}\mathrm{d}t},\ \ (u,v)\in[0,1]^2,   
$$
while the $B_{i,t}(\cdot)$'s ($i=1,2)$ are  generated as  i.i.d. Brownian bridges, independent for different $i$.  
Notice that, in both cases above, $\delta=0$ corresponds to $H_0$ while $\delta>0$ corresponds to $H_1$.

All curves were approximated using $T=21$ equidistant points $\tau_1,\tau_2,\ldots,\tau_{21}$ in the unit interval $\mathcal{I}$ and  transformed into functional objects using the Fourier basis with $21$ basis functions.
Functional time series of length  $n_1 = n_2= 200$   are   then generated and  testing the null hypothesis  $ H_0: \mathcal{C}_{1,0}=\mathcal{C}_{2,0}$ is considered using the $T_M$ test investigated  Section 3.2.   All bootstrap calculations are  based on $B = 1000$ bootstrap replicates,  $R = 1000$ model repetitions have been considered and  a range of different block sizes have been used.  Since  $n_1=n_2$ we  set for simplicity 
 $b = b_1 = b_2$.  
 
 Regarding the selection of $b$ we mention the following. As an inspection of the proof of Theorem 2.1 shows, the MBB estimator  of the distribution of interest  also  delivers   a lag-window type estimator of the covariance operator $\Gamma_0$ of  the limiting Gaussian process $ \mathcal{Z}_0$ using implicitly  the  Bartlett lag-window with ``truncation lag" the block size $b$; see also equation (\ref{eq.mean-MBB-star}). Viewing   the choice of $b$  as the selection of the truncation lag in the aforementioned  lag window type estimator, allows for the use of some results available  in the literature in order  to select $b$. To elaborate,  the choice of the truncation lag  in the functional set-up   has been  discussed in  Horv\'ath {\em et al.} (2016)  and  Rice and  Shang (2017), where different procedures to select this parameter have been investigated. In our context, we found  
 the   simple rule  proposed by Rice and Shang (2017) quite effective according to   
 which   the block length  $ b$ is set equal to the smallest integer larger or equal   to $ n^{0.3}$. Various choices of the block length $b$ have been considered in our simulations.

 The $T_M$ test has  been applied using three  standard nominal levels $\alpha=0.01,$ $0.05$ and $0.10.$ 
Notice that  $\delta=0$  corresponds to the null hypothesis while to investigate the power behavior of the  test  we set $\delta=0$ for the first functional time series and allow for $\delta\in\{0.2,0.5,0.8\}$ for the second and  for each of the two different models considered. The results obtained for different values of the block size $b$ using the FAR model (\ref{model:FAR1}) as well as the   FMA model (\ref{model:FMA1a}) are shown in Table~\ref{tab:apotelesmataNEW}.  As it is seen from  this table, the MBB based testing procedure retains the nominal level  with  good  size results for both dependence structures  considered. Furthermore, the power of the $T_M$ test increases as the deviations from the null increase  and reaches  high values   
for the large values of the deviation parameter $\delta$ considered.

\begin{table}[htbp]
  \centering
      \begin{tabular}{|ccc|ccccc|}
    \hline
& & & \multicolumn{5}{c|}{Block Size, $b$=}\\
& $\delta$ &$\alpha$ & 2& 4& 6& 8& 10\\
\hline
FAR (1)&$0$&0.01 &0.011&0.022&0.014&0.021&0.018\\
& & 0.05&0.050 &0.062&0.063&0.083&0.076\\
& & 0.10&0.108 &0.123&0.108&0.132&0.125\\
& 0.2& 0.01&0.025 &0.018&0.020&0.025&0.026\\
& & 0.05&0.089 &0.093&0.085&0.081&0.089\\
& & 0.10&0.151 &0.171&0.150&0.156&0.151\\
& 0.5& 0.01&0.593 &0.495&0.411&0.381&0.375\\
& & 0.05&0.776 &0.731&0.698&0.676&0.672\\
& & 0.10&0.839 &0.813&0.794&0.788&0.791\\
& 0.8& 0.01&1.000&1.000&1.000&0.997&0.989\\
& & 0.05&1.000 &1.000&1.000&1.000&1.000\\
& & 0.10&1.000 &1.000&1.000&1.000&1.000\\
FAM (1)&$0$&0.01 &0.012&0.013&0.014&0.013&0.015\\
& & 0.05&0.065 &0.073&0.060&0.054&0.071\\
& & 0.10&0.121 &0.108&0.118&0.116&0.127\\
& 0.2& 0.01&0.015 &0.022&0.019&0.024&0.016\\
& & 0.05&0.055 &0.076&0.065&0.079&0.062\\
& & 0.10&0.1114 &0.130&0.119&0.123&0.122\\
& 0.5& 0.01&0.148 &0.125&0.143&0.121&0.131\\
& & 0.05&0.339 &0.239&0.330&0.292&0.289\\
& & 0.10&0.479 &0.421&0.468&0.412&0.418\\
& 0.8& 0.01&0.074&0.695&0.689&0.693&0.681\\
& & 0.05&0.920 &0.889&0.899&0.887&0.900\\
& & 0.10&0.957 &0.944&0.941&0.949&0.957\\
\hline
     \end{tabular}
\caption{Empirical size and power of the  $T_M$ test  using  bootstrap critical values.}
  \label{tab:apotelesmataNEW}
\end{table}

\subsection{\sc Cyprus Daily Temperature Data}
In this section, the  bootstrap based  $T_M$ testing is applied to a real-life data set which consists  of  daily temperatures recorded in $15$ minutes intervals in Nicosia,  Cyprus, i.e., there are $96$ temperature measurements for each day. Sample A and Sample B consist of the daily temperatures recorded in Summer~2007 (01/06/2007-31/08/2007) and Summer~2009 (01/06/2009-31/08/2009) respectively. The measurements have been transformed into functional objects using the Fourier basis with 21 basis functions. All curves are rescaled in order to be defined in the interval $\mathcal{I}=[0,1]$. Figure~\ref{fig:temperature} shows the estimated lag-zero autocovariance  kernels  $\widehat{c}_i(u,v)=n_i^{-1}\sum_{t=1}^{n_i}(X_{i,t}(u)-\overline{X}_i(u))(X_{i,t}(v)-\overline{X}_i(v)) $, $(u,v) \in  \mathcal{I} \times \mathcal{I}$, associated with the lag-zero autocovariance operators for  the temperature curves of the summer 2007 ($i=1$) and of the summer  2009 ($i=2$). We are interested in testing whether 
  the covariance structure of the daily  temperature curves  of the two summer periods  is the same, a question which can be important in the context of investigating the changing behavior of the Mediterranean climate. Furthermore, such a question could also arise if one is concerned with the stationarity behavior of the centered time series of temperature curves. The bootstrap $p$-values of the MBB-based  $T_M$ test using $B=1000$ bootstrap replicates  and  for a selection of different  block sizes $b=b_1=b_2$, are equal to 0.016 ($b=3$), 0.015 ($b=4$), 0.033 ($b=5$) and 0.030 ($b=6$).  Notice  that in this example,  $n_1=n_2=92$  and that, for  
this sample size, the value of $b=4$ is the one chosen  by the simple selection rule discussed in the previous section.   As it is evident from these results, the bootstrap $p$-values of the MBB-based test are quite small and lead to a rejection of 
$H_0$, for instance at the commonly used 5\% level. 

\begin{figure}[h]
\centering
\includegraphics[trim = 5mm 10mm 5mm 10mm, clip=true,width=0.9\textwidth]{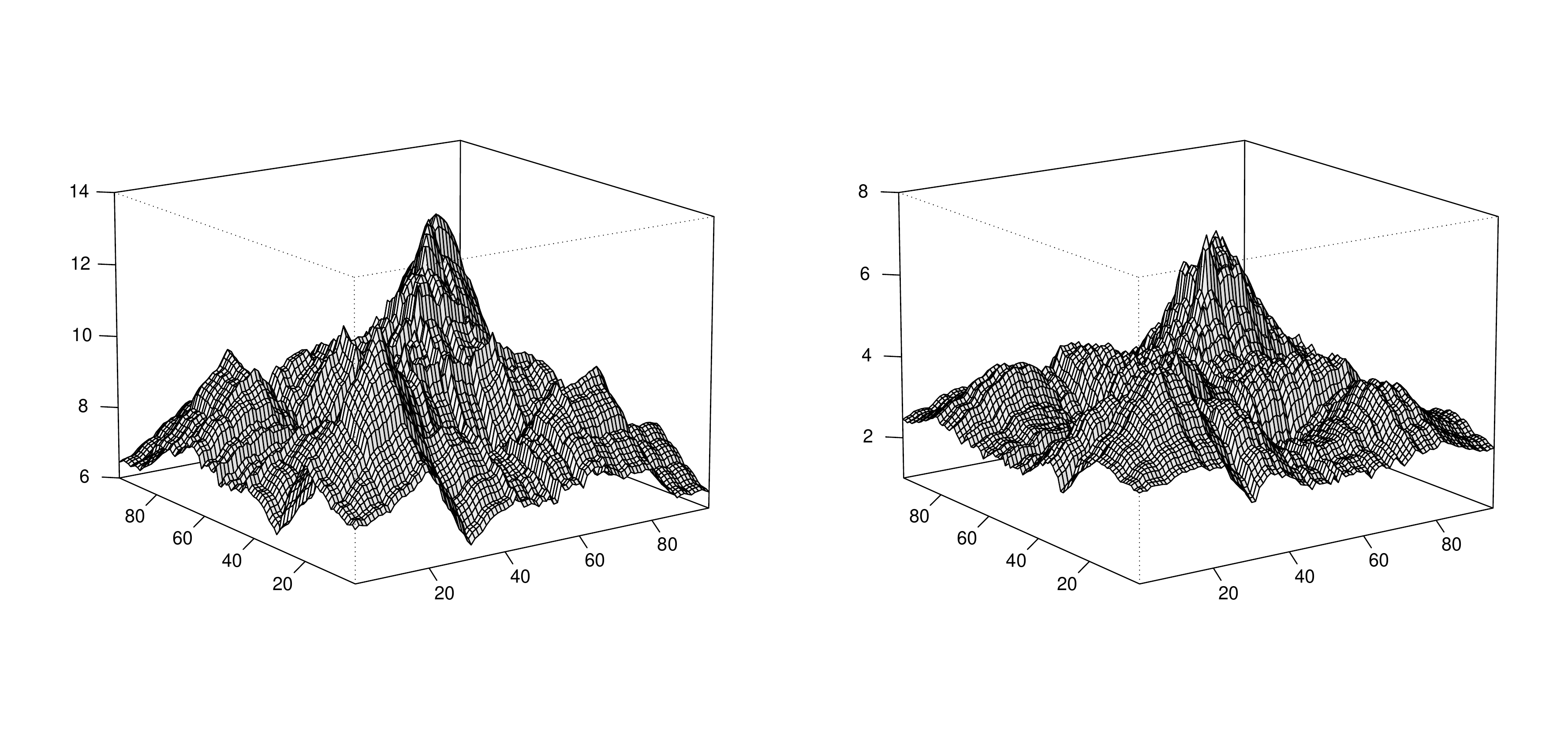}
\caption{Estimated lag-zero autocovariance  kernels   of the temperature curves:  Summer 2007 (left panel) and summer 2009 (right panel).}
\label{fig:temperature}
\end{figure}

To  see were the differences between the temperatures in the two summer periods come from and 
to better interpret   the test results, Figure \ref{fig:grayplotNew}  presents a contour  plot of the estimated  squared differences 
$ |\widehat{c}_1(u,v)-\widehat{c}_2(u,v)|^2$
for different values of $(u,v)$ in the plane $ [0,1]^2$.  
Note that the Hilbert-Schmidt distance  $\| \widehat{\mathcal{C}}_{1,0} -\widehat{\mathcal{C}}_{2,0}\|_{HS}$   appearing in the  test statistic $ T_M$  can be approximated by the discretized quantity $ \sqrt{ L^{-2}\sum_{i=1}^L\sum_{j=1}^L  
|\widehat{c}_1(u_i,v_j)-\widehat{c}_2(u_i,v_j)|^2}$, where  $L=96$ is the number of equidistant  time points in the interval $[0,1]$ used and at which the temperature measurements are recorded. Large values of $  |\widehat{c}_1(u_i,v_j)-\widehat{c}_2(u_i,v_j)|^2$ (i.e., dark gray regions in Figure \ref{fig:grayplotNew}) contribute strongly to the value of the test statistic $T_M$ and pinpoint to regions   where large differences between the corresponding lag-zero autocovariance operators occur. Taking into account the symmetry of the covariance kernel  $c(\cdot,\cdot)$, Figure \ref{fig:grayplotNew} is very informative. It shows  that the main  differences between the two covariance operators  are concentrated between the time regions 3:00am to 6:00am and  3:00pm to 8:00pm of the daily temperature curves,   with the  
strongest  contributions to the test statistic  being due to the largest differences recorded  around 4:00 to 4:30 in the morning and  6:30 to  7:30 in the evening.



\begin{figure}
\centering
\includegraphics[trim = 5mm 10mm 5mm 10mm, clip=true,width=0.7\textwidth, height=0.45 \textheight]{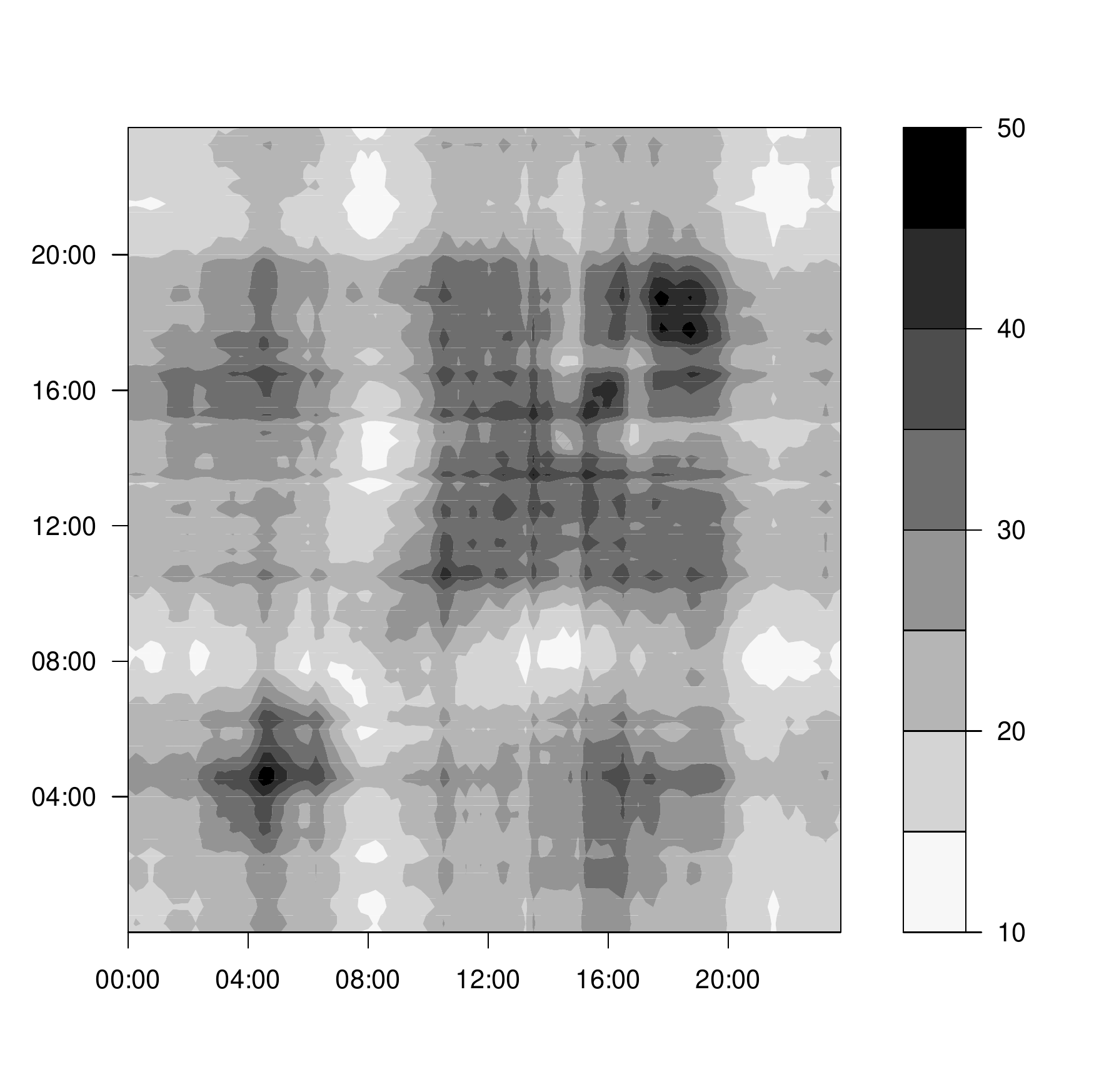}
\caption{Contour plot of the estimated differences $ |\widehat{c}_1(u_i,v_j)-\widehat{c}_2(u_i,v_j)|^2$  for   $(i,j) \in \{1,2, \ldots, 96\}$.}
\label{fig:grayplotNew}
\end{figure}

\section{\sc Appendix : Proofs}
\label{sec:proofs}
In the following we assume, without loss of generality, that $\mu=0$ and we consider the case $h=0$ only.
Furthermore,  we let $\widehat{\widetilde{\C}}_0=n^{-1}\sum_{t=1}^{n}X_t\otimes X_t,$ $Z_t=X_t\otimes X_t-\C_0,$ $\widehat{Z}_t=X_t\otimes X_t-\widehat{\widetilde{\C}}_0,$ $\widetilde{Z}_t=X_t\otimes X_t,$ $Z_{t,m}=X_{t,m}\otimes X_{t,m}-\C_0,$ $Z^*_t=X^*_t\otimes X^*_t$ and $\widehat{Z}^*_t=X^*_t\otimes X^*_t-\widehat{\widetilde{\C}}_0.$ Also, we denote by $Z_t(u,v)$ the kernel of the integral operator $Z_t,$ i.e., $Z_t(u,v)=X_t(u)X_t(v)-c_0(u,v),$ where $c_0(u,v)=\E [X_t(u)X_t(v)],$ and by $Z_{t,m}(u,v)$ the kernel of the integral operator $Z_{t,m},$ i.e., $Z_{t,m}(u,v)=X_{t,m}(u)X_{t,m}(v)-c_0(u,v).$

We first fix some notation and present two basic lemmas  which will be used in  the proofs. Towards this note first that we repeatedly use the fact that, by stationarity, $\E\|X_{t,m}-X_t\|^p=\E\|X_{0,m}-X_0\|^p$ and $\E\|X_{t,m}\|^p=\E\|X_t\|^p=\E\|X_0\|^p$ for $p\in\mathbb{N}$ and for all $t\in \mathbb{Z}.$ Also note that Kokoszka and Reimherr (2013) proved that the $L^4$-$m$-approximability of $\mathds{X}$ implies that the tensor product process $\{X_t\otimes X_t,\,t\in\mathbb{Z}\}$ is $L^2$-$m$-approximable. \\

For $X_{t,m}\otimes X_{t,m}$ the $m$-dependent approximation of $X_t\otimes X_t$, we, therefore, have
\begin{equation}\label{eq:L2}
\sum_{m=1}^{\infty}\bigg(\E\|X_t\otimes X_t-X_{t,m}\otimes X_{t,m}\|_{HS}^2\bigg)^{1/2}<\infty.
\end{equation}
Furthermore, since $\|X_0\otimes X_t\|_{HS}=\|X_0\|\|X_t\|$ for all $t\in \mathbb{Z},$ and using Cauchy-Schwarz's inequality, we get, for all $t\in \mathbb{Z},$
\begin{align*}
\E\|X_t\otimes X_t -X_{t,m}\otimes X_{t,m}\|_{HS}^2
&\leq 2\E\|X_t\otimes (X_t-X_{t,m})\|_{HS}^2+2\E\| (X_t-X_{t,m})\otimes X_{t,m}\|_{HS}^2\\&\leq 4(\E\|X_t\|^4)^{1/2}(\E\|X_t-X_{t,m}\|^4)^{1/2}.
\end{align*}
Therefore, by Assumption~\ref{as:strongL4}, we get, for all $t\in \mathbb{Z},$
\begin{equation}\label{eq:strongL2}
\lim_{m\to\infty} m\left(\E\|X_t\otimes X_t -X_{t,m}\otimes X_{t,m}\|_{HS}^2\right)^{1/2}\leq 2(\E\|X_t\|^4)^{1/4}\lim_{m\to\infty}m(\E\|X_t-X_{t,m}\|^4)^{1/4}=0
\end{equation}
and by the same arguments,
\begin{align*}
\|\E[X_0\otimes X_t]\|_{HS}=\|\E[X_0\otimes (X_t-X_{t,t}]\|_{HS}&\leq \left(\E\|X_0\|^2_{HS}\right)^{1/2}\left(\E\|X_0-X_{0,t}\|^2_{HS}\right)^{1/2}\\&\leq\left(\E\|X_0\|^2_{HS}\right)^{1/2}\left(\E\|X_0-X_{0,t}\|^4_{HS}\right)^{1/4}.
\end{align*}
Therefore, the $L^4$-$m$-approximability assumption implies that $\sum_{t\in\mathbb{Z}}\|\E[X_0\otimes X_t]\|_{HS}<\infty$.
\par
\medskip
To prove Theorem~\ref{thm:CLTmbb}, we establish below Lemma~\ref{lemma:sigklisidiasporasMBB} and Lemma~\ref{lemma:MBBtight}. Their proofs are given in the supplementary material.

\begin{lemma}\label{lemma:sigklisidiasporasMBB}$\,$
Let $g_b(\cdot)$ be a non-negative, continuous and bounded function defined on $\mathbb{R}$, satisfying $g_b(0)=1$, $g_b(u)=g_b(-u)$, $g_b(u)\leq 1$ for all $u$, $g_b(u)=0,$ if $|u|>c,$ for some $c>0.$ Assume that for any fixed $u$, $g_b(u)\to 1$ as $n\to\infty.$ Suppose that the process  $ \mathds{X}$   satisfies Assumption~\ref{as:strongL4} and  that $b=b(n)$ is a sequence of integers such that $b^{-1}+bn^{-1/3}=o(1)$ as $n\to\infty.$ Then, as $n\to\infty$,
$$
\big\|\sum_{s=-b+1}^{b-1}g_b(s)\hat{\Gamma}_s-\sum_{s=-\infty}^{\infty}\E[ Z_0 \otimes Z_s]\big\|_{HS}=o_p(1),
$$
where
$\hat{\Gamma}_s=\frac{1}{n}\sum_{t=1}^{n-s}\hat{Z}_t\otimes \hat{Z}_{t+s}$ for $0\leq s\leq b-1$ and $\hat{\Gamma}_s=\frac{1}{n}\sum_{t=1}^{n+s}\hat{Z}_{t-s}\otimes \hat{Z}_t$ for $-b+1\leq s<0.$
\end{lemma}

\begin{lemma}\label{lemma:MBBtight}
Let $g_b(\cdot)$ be a non-negative, continuous and bounded function satisfying the conditions of Lemma~\ref{lemma:sigklisidiasporasMBB}.
Suppose that $ \mathds{X}$ satisfies Assumption~\ref{as:strongL4} and that $b=b(n)$ is a sequence of integers such that $b^{-1}+bn^{-1/2}=o(1)$ as $n\to\infty.$ Then, as $n\to\infty$,
$$
\sum_{s=-b+1}^{b-1}g_b(s)\dfrac{1}{n}\sum_{t=1}^{n-|s|}\iint Z_t(u,v)Z_{t+|s|}(u,v)\mathrm{d}u\mathrm{d}v \overset{P}\to \sum_{s=-\infty}^{\infty}\E \iint Z_0(u,v)Z_s(u,v)\mathrm{d}u\mathrm{d}v.
$$
\end{lemma}
\medskip

\noindent \begin{proof}[Proof of Theorem~\ref{thm:CLTmbb}]
By the triangle inequality and Theorem~$3$ of Kokoszka and Reimherr (2013), the assertion of the theorem is established if we show that, as $n\to\infty,$
\begin{equation}
\label{expr:CLT}
\sqrt{n}(\hat{\C}_{0}^*-\E^*(\hat{\C}_{0}^*)) \Rightarrow  \mathcal{Z}_0,
\end{equation}
in probability, where $\mathcal{Z}_0$ is a mean zero Gaussian Hilbert Schmidt operator with covariance operator given by
$$
\Gamma_0=\E[Z_1\otimes Z_1]+2\sum_{s=2}^{\infty}\E[Z_1\otimes Z_s].
$$
Using Theorem~1 of Horv\'{a}th {\em et al.} (2013), 
we get
\begin{align}
\sqrt{n}(\hat{\C}_{0}^*-&\E^*(\hat{\C}_{0}^*))\nonumber \\ &
=\dfrac{1}{\sqrt{n}}\sum_{t=1}^{n}\bigg[X_t^*\otimes X_t^*-\E^*(X_t^*\otimes X_t^*)-\overline{X}_n\otimes (X_t^*-\E^*(X_t^*))- (X_t^*-\E^*(X_t^*))\otimes \overline{X}_n \bigg]\nonumber\\
& =\dfrac{1}{\sqrt{n}}\sum_{t=1}^{n}[Z_t^*-\E^*(Z_t^*)]+O_P(1/\sqrt{n}).\nonumber
\end{align}
Also note that 
\begin{align*}
\dfrac{1}{\sqrt{n}}\sum_{t=1}^{n}[Z_t^*-\E^*(Z_t^*)] & =\dfrac{1}{\sqrt{k}}\sum_{t=1}^{k}\left(\dfrac{1}{\sqrt{b}} \sum_{i=1}^{b}\big(Z^*_{(t-1)b+i}-\E^*(Z^*_{(t-1)b+i})\big)\right)\\
&  =\dfrac{1}{\sqrt{k}}\sum_{t=1}^{k}\widehat{Y}_t^*,
\end{align*}
with an obvious notation for $\widehat{Y}^*_t,\,t=1,2,\ldots,k$. Recall that due to the block bootstrap resampling scheme, the random variables $\widehat{Y}^*_t,\,t=1,2,\ldots,k,$ are i.i.d. Therefore to prove~\eqref{expr:CLT}, it suffices by Lemma~$5$ of Kokoszka and Reimherr (2013), to prove  that,
\begin{enumerate}[label=(\roman*)]
  \item $\left\langle \dfrac{1}{\sqrt{k}}\sum_{t=1}^{k}\widehat{Y}_t^*,y\right\rangle_{HS}\overset{d}\to N(0,\sigma^2(y))$ for every Hilbert Schmidt operator $y$ acting on $L^2,$ \label{condition1}
\end{enumerate}
and that
\begin{enumerate}[label=(\roman*)]
  \setcounter{enumi}{1}
  \item  $\lim_{n\to\infty}\E^*\left\|\dfrac{1}{\sqrt{k}}\sum_{t=1}^{k}\hat{Y}_t^*\right\|^2_{HS}$ exists and is finite.\label{condition2}
\end{enumerate}

To establish assertion~\ref{condition1}, we first prove that, as $n\to\infty$,
\begin{equation}
\label{eq:condition1i}
\Var^*\left(\left\langle \dfrac{1}{\sqrt{k}}\sum_{t=1}^{k}\widehat{Y}_t^*,y\right\rangle_{HS}\right)\overset{P}\to\sigma^2(y).
\end{equation}
Consider~\eqref{eq:condition1i} and notice that
\begin{equation}
\Var^*\left(\left\langle \dfrac{1}{\sqrt{k}}\sum_{t=1}^{k}\widehat{Y}_t^*,y\right\rangle_{HS}\right)= \Var^*\left(\langle \widehat{Y}_1^*,y\rangle_{HS}\right)= \E^*\left[\left\langle \dfrac{1}{\sqrt{b}}\sum_{t=1}^{b}(Z^*_t-\E^*(Z^*_t)),y\right\rangle_{HS}\right]^2\label{eq:diasporaW1}.
\end{equation}
Let $N=n-b+1,$ $\widetilde{Y}_t=b^{-1/2}(\widetilde{Z}_t+\widetilde{Z}_{t+1}+\ldots+\widetilde{Z}_{t+b-1}),$ $t=1,2,\ldots,N$ and $\widetilde{Y}^*_t=b^{-1/2}\sum_{i=1}^{b}Z^*_{(t-1)b+i},$ $t=1,2,\ldots,k.$ Since $n/N\to1$ as $n\to\infty,$ in the following we will occasionally replace $1/N$ by $1/n.$
Notice that,
\begin{align}
\E^*\left(\left\langle\dfrac{1}{\sqrt{b}}\sum_{t=1}^{b}Z^*_t,y\right\rangle_{HS}\right)&=\E^*(\widetilde{Y}^*_1)=\dfrac{1}{N}\sum_{t=1}^{N}\langle \widetilde{Y}_t,y\rangle_{HS}\nonumber\\&=\dfrac{\sqrt{b}}{N}\left[\sum_{t=1}^{n}\langle \tilde{Z}_t,y\rangle_{HS}-\sum_{i=1}^{b-1}\left(1-\dfrac{i}{b}\right)[\langle \widetilde{Z}_i,y\rangle_{HS}+\langle \widetilde{Z}_{n-i+1} ,y\rangle_{HS}]\right]\nonumber\\&
=\langle\sqrt{b}\,\hat{\tilde{\C}}_{n},y\rangle-\dfrac{\sqrt{b}}{N}\left[\sum_{i=1}^{b-1}\left(1-\dfrac{i}{b}\right)[\langle \widetilde{Z}_i,y\rangle_{HS}+\langle \widetilde{Z}_{n-i+1} ,y\rangle_{HS}]\right]\label{eq:anamenomeniW1}.
\end{align}
Therefore,
\begin{align}
\Var^*&\left(\left\langle \dfrac{1}{\sqrt{k}}\sum_{t=1}^{k}\widehat{Y}_t^*,y\right\rangle_{HS}\right)\nonumber\\&=  \E^*\left[\left\langle \dfrac{1}{\sqrt{b}}\sum_{t=1}^{b}\widehat{Z}^*_t,y\right\rangle_{HS}+\dfrac{\sqrt{b}}{N}\left[\sum_{i=1}^{b-1}\left(1-\dfrac{i}{b}\right)[\langle \widetilde{Z}_i,y\rangle_{HS}+\langle \widetilde{Z}_{n-i+1} ,y\rangle_{HS}]\right]\right]^2\nonumber\\&=
\E^*\left[\left\langle \dfrac{1}{\sqrt{b}}\sum_{t=1}^{b}\widehat{Z}^*_t,y\right\rangle_{HS}\right]^2+ \left[\dfrac{\sqrt{b}}{N}\left[\sum_{i=1}^{b-1}\left(1-\dfrac{i}{b}\right)[\langle \widetilde{Z}_i,y\rangle_{HS}+\langle \widetilde{Z}_{n-i+1} ,y\rangle_{HS}]\right]\right]^2 \nonumber\\&\qquad +2\left[\dfrac{\sqrt{b}}{N}\left[\sum_{i=1}^{b-1}\left(1-\dfrac{i}{b}\right)[\langle \widetilde{Z}_i,y\rangle_{HS}+\langle \widetilde{Z}_{n-i+1} ,y\rangle_{HS}]\right]\right]\E^*\left[\left\langle \dfrac{1}{\sqrt{b}}\sum_{t=1}^{b}\widehat{Z}^*_t,y\right\rangle_{HS}\right]\nonumber\\&= \E^*\left[\left\langle \dfrac{1}{\sqrt{b}}\sum_{t=1}^{b}\hat{Z}^*_t,y\right\rangle_{HS}\right]^2+O_P(b^3/n^2).\label{eq:diaspora}
\end{align}
Let $\widehat{Y}_t=b^{-1/2}(\widehat{Z}_t+\widehat{Z}_{t+1}+\ldots+\widehat{Z}_{t+b-1}),$ $t=1,2,\ldots,N.$ Since,
\begin{align*}
\E^*&\left[\left\langle \dfrac{1}{\sqrt{b}}\sum_{t=1}^{b}\widehat{Z}^*_t,y\right\rangle_{HS}\right]^2=
\dfrac{1}{N}\sum_{t=1}^{N}\langle \widehat{Y}_t,y\rangle_{HS}^2\\&=
\dfrac{1}{N}\sum_{t=1}^{n}\langle \widehat{Z}_t,y\rangle_{HS}\langle \widehat{Z}_t,y\rangle_{HS}\\&\;+
\sum_{i=1}^{b-1}\left(1-\dfrac{i}{b}\right)
\dfrac{1}{N}\sum_{t=1}^{n-i}[\langle \widehat{Z}_t,y\rangle_{HS}\langle \widehat{Z}_{t+i},y\rangle_{HS}+\langle \widehat{Z}_{t+i},y\rangle_{HS}\langle \widehat{Z}_t,y\rangle_{HS}]\\&
\;-\dfrac{1}{N}\sum_{i=1}^{b-1}\left(1-\dfrac{i}{b}\right)[\langle \widehat{Z}_i,y \rangle_{HS} \langle \widehat{Z}_i,y \rangle_{HS} +\langle \widehat{Z}_{n-i+1},y \rangle_{HS} \langle \widehat{Z}_{n-i+1},y \rangle]_{HS}\\&
\;-\dfrac{1}{N}\sum_{i=1}^{b-1}\sum_{j=1}^{b-t}\left(1-\dfrac{j+i}{b}\right)[\langle \hat{Z}_j,y\rangle_{HS}\langle \widehat{Z}_{j+i},y\rangle_{HS}+\langle \widehat{Z}_{n-j+1-i},y\rangle_{HS}\langle \widehat{Z}_{n-j+1},y\rangle_{HS}\\&\hspace{105pt}
+\langle \widehat{Z}_{j+i},y\rangle_{HS}\langle \widehat{Z}_{j},y\rangle_{HS}+\langle \widehat{Z}_{n-j+1},y\rangle_{HS}\langle \widehat{Z}_{n-j+1-i},y\rangle_{HS}],
\end{align*}
we get, using \eqref{eq:diaspora}, 
\begin{align*}
&\Var^*\left(\left\langle \dfrac{1}{\sqrt{k}}\sum_{t=1}^{k}\widehat{Y}_t^*,y\right\rangle_{HS}\right)\\&=
\dfrac{1}{N}\sum_{t=1}^{n}\langle \widehat{Z}_t,y\rangle_{HS}\langle \widehat{Z}_t,y\rangle_{HS}+
\sum_{i=1}^{b-1}\left(1-\dfrac{i}{b}\right)
\dfrac{1}{N}\sum_{t=1}^{n-i}[\langle \widehat{Z}_t,y\rangle_{HS}\langle \widehat{Z}_{t+i},y\rangle_{HS}+\langle \widehat{Z}_{t+i},y\rangle_{HS}\langle \widehat{Z}_t,y\rangle_{HS}]\\&\hspace{280pt}+O_P(b/n)+O_P(b^2/n)+O_P(b^3/n^2).
\end{align*}
Therefore, 
\begin{align}
&\Var^*\left(\left\langle \dfrac{1}{\sqrt{k}}\sum_{t=1}^{k}\widehat{Y}_t^*,y\right\rangle_{HS}\right)\nonumber\\& =\dfrac{1}{N}\sum_{t=1}^{n}\langle \widehat{Z}_t\otimes \widehat{Z}_t,y\otimes y\rangle_{HS}+
\sum_{i=1}^{b-1}\left(1-\dfrac{i}{b}\right)
\dfrac{1}{N}\sum_{t=1}^{n-i}[\langle \widehat{Z}_t\otimes \widehat{Z}_{t+i},y\otimes y\rangle_{HS}+\langle \widehat{Z}_{t+i}\otimes \widehat{Z}_{t},y\otimes y\rangle_{HS}]\nonumber\\&\hspace{340pt}+O_P(b^2/n).\label{eq:VarY1y}
\end{align}
Let $g_b(i)=\left(1-\frac{|i|}{b}\right)$ in Lemma~\ref{lemma:sigklisidiasporasMBB}, and use the triangular inequality to get
\begin{align*}
&\left|\left\langle\dfrac{1}{N}\sum_{t=1}^{n} \widehat{Z}_t\otimes \widehat{Z}_t+\sum_{i=1}^{b-1}\left(1-\dfrac{i}{b}\right)
\dfrac{1}{N}\sum_{t=1}^{n-i}[\widehat{Z}_t\otimes \widehat{Z}_{t+i}+\widehat{Z}_{t+i}\otimes \widehat{Z}_{t}]-\sum_{t=-\infty}^{\infty}\E[Z_0\otimes Z_t],y\otimes y\right\rangle_{HS}\right|\\&\hspace{2pt}\leq \left\|\dfrac{1}{N}\sum_{t=1}^{n}\widehat{Z}_t\otimes \widehat{Z}_t+\sum_{i=1}^{b-1}\left(1-\dfrac{i}{b}\right)
\dfrac{1}{N}\sum_{t=1}^{n-i}[ \widehat{Z}_t\otimes \widehat{Z}_{t+i}+ \widehat{Z}_{t+i}\otimes \widehat{Z}_{t}]-\sum_{t=-\infty}^{\infty}\E[Z_0\otimes Z_t]\right\|_{HS}\big\|y\otimes y\big\|_{HS}\\&=o_p(1).
\end{align*}
Therefore, and using $\langle Z_0\otimes Z_t,y\otimes y\rangle_{HS}=\langle Z_0,y\rangle_{HS}\langle Z_t,y\rangle_{HS}$, we get from~\eqref{eq:VarY1y}, as $n\to\infty,$
\begin{equation}
\Var^*\left(\left\langle \dfrac{1}{\sqrt{k}}\sum_{t=1}^{k}\widehat{Y}_t^*,y\right\rangle_{HS}\right)\overset{P}\to\left\langle\sum_{t=-\infty}^{\infty}\E[Z_0\otimes Z_t],y\otimes y\right\rangle_{HS}\label{eq:sigklisiVarw1}=\langle\Gamma_0,y\otimes y\rangle_{HS}=\sigma^2(y).
\end{equation}

We next establish the asymptotic normality stated in (i). Since  $\langle \widehat{Y}_t^*,y\rangle_{HS},\,t=1,2,\ldots,k$ are i.i.d. real valued random variables, we show  that Lindeberg's condition
is satisfied, i.e., for every $\varepsilon>0,$ as $n\to\infty,$
\begin{equation}
\label{eq:Lindeberg}
\dfrac{1}{\tau_k^{*2}}\sum_{t=1}^{k}\E^*\left[\big(\langle \widehat{Y}_t^*,y\rangle_{HS}-\E^*(\langle \widehat{Y}_t^*,y\rangle_{HS})\big)^2 \mathds{1}\big(|\langle \widehat{Y}_t^*,y\rangle_{HS}-E^*(\langle \widehat{Y}_t^*,y\rangle_{HS})|>\varepsilon \tau_k^*\big)\right]=o_p(1),
\end{equation}
where $\mathds{1}_A(x)$ denotes the indicator function of the set $A$ and
\begin{equation}\label{eq:tLindeberg}
\tau_k^{*2}=\sum_{t=1}^{k}\Var^*(\langle \widehat{Y}_t^*,y\rangle_{HS})=k\Var^*(\langle \widehat{Y}_1^*,y\rangle_{HS}).
\end{equation}
To establish~\eqref{eq:Lindeberg}, and because of~\eqref{eq:sigklisiVarw1} and~\eqref{eq:tLindeberg}, it suffices to show that, for any $\delta>0,$ as $n\to\infty,$
\begin{equation}
\label{eq:sufficeLindeberg}
P\left(\dfrac{1}{k}\sum_{t=1}^{k}\E^*\left[(\langle \widehat{Y}_t^*,y\rangle_{HS}-\E^*(\langle \widehat{Y}_t^*,y\rangle_{HS}))^2 \mathds{1}(|\langle \widehat{Y}_t^*,y\rangle_{HS}-E^*(\langle \widehat{Y}_t^*,y\rangle_{HS})|>\varepsilon \tau_k^*)\right]>\delta\right)\to0.
\end{equation}
Towards this, notice first that, for any two random variables $X$ and $Y$ and any $\eta>0,$
\begin{align}
\E[|X+Y|^2&\mathds{1}(|X+Y|>\eta)]\nonumber\\&\leq
4\left[\E|X|^2\mathds{1}(|X|>\eta/2)+\E|Y|^2\mathds{1}(|Y|>\eta/2)\right];\label{eq:Lahiri}
\end{align}
see Lahiri $(2003)$, p. $56$. Since the random variables $\langle \widehat{Y}_t^*,y\rangle_{HS}$ are i.i.d., we get using expression~\eqref{eq:anamenomeniW1} and Markov's inequality that, as $n\to\infty,$
\begin{align}
&P\left(\dfrac{1}{k}\sum_{t=1}^{k}\E^*\left[(\langle \widehat{Y}_t^*,y\rangle_{HS}-\E^*(\langle \widehat{Y}_t^*,y\rangle_{HS}))^2 \mathds{1}(|\langle \widehat{Y}_t^*,y\rangle_{HS}-\E^*(\langle \widehat{Y}_t^*,y\rangle_{HS})|> \varepsilon \tau_k^*)\right]>\delta\right) \nonumber \\&
\leq\delta^{-1}\E\left\{\E^*\left[(\langle \widehat{Y}_1^*,y\rangle_{HS}-\E^*(\langle \widehat{Y}_1^*,y\rangle_{HS}))^2 \mathds{1}(|\langle \widehat{Y}_1^*,y\rangle_{HS}-\E^*(\langle \widehat{Y}_1^*,y\rangle_{HS})|> \varepsilon \tau_k^*)\right]\right\}\nonumber \\&
=\delta^{-1}\E\Bigg\{\E^*\Bigg[\left(\left\langle \dfrac{1}{\sqrt{b}}\sum_{t=1}^{b}\widehat{Z}^*_t,y\right\rangle_{HS}+\dfrac{\sqrt{b}}{N}\left[\sum_{i=1}^{b-1}\left(1-\dfrac{i}{b}\right)[\langle \widetilde{Z}_i,y\rangle_{HS}+\langle \widetilde{Z}_{n-i+1} ,y\rangle_{HS}]\right]\right)^2  \nonumber \\&\hspace{35pt}\times\mathds{1}\left(\left|\left\langle \dfrac{1}{\sqrt{b}}\sum_{t=1}^{b}\widehat{Z}^*_t,y\right\rangle_{HS}+\dfrac{\sqrt{b}}{N}\left[\sum_{i=1}^{b-1}\left(1-\dfrac{i}{b}\right) [\langle \widetilde{Z}_i,y\rangle_{HS}+\langle \widetilde{Z}_{n-i+1} ,y\rangle_{HS}]\right]\right|> \varepsilon \tau_k^*\right)\Bigg]\Bigg\}\nonumber \\&
= \delta^{-1}\E\Bigg[\dfrac{1}{N}\sum_{t=1}^{N}\left(\langle \widehat{Y}_t,y\rangle_{HS}+\dfrac{\sqrt{b}}{N}\left[\sum_{i=1}^{b-1}\left(1-\dfrac{i}{b}\right)[\langle \widetilde{Z}_i,y\rangle_{HS}+\langle \widetilde{Z}_{n-i+1} ,y\rangle_{HS}]\right]\right)^2 \nonumber \\&\hspace{35pt}\times
\mathds{1}\left(\left|\langle \widehat{Y}_t,y\rangle_{HS}+\dfrac{\sqrt{b}}{N}\left[\sum_{i=1}^{b-1}\left(1-\dfrac{i}{b}\right)[\langle \widetilde{Z}_i,y\rangle_{HS}+\langle \widetilde{Z}_{n-i+1} ,y\rangle_{HS}]\right]\right|> \varepsilon \tau_k^*\right)\Bigg]
\nonumber \\&
\leq 4\delta^{-1}\Bigg[\E(\langle \widehat{Y}_1,y\rangle_{HS}^2)\mathds{1}(|\langle \widehat{Y}_1,y\rangle_{HS}|> \varepsilon \tau_k^*/2)+\E\left(\dfrac{\sqrt{b}}{N}\sum_{i=1}^{b-1}\left(1-\dfrac{i}{b}\right)[\langle \widetilde{Z}_i,y\rangle_{HS}+\langle \widetilde{Z}_{n-i+1} ,y\rangle_{HS}]\right)^2\nonumber \\&\hspace{130pt}\times\mathds{1}(\left|\left(\dfrac{\sqrt{b}}{N}\sum_{i=1}^{b-1}\left(1-\dfrac{i}{b}\right)[\langle \widetilde{Z}_i,y\rangle_{HS}+\langle \widetilde{Z}_{n-i+1} ,y\rangle_{HS}]\right)\right|> \varepsilon \tau_k^*/2)\Bigg]\nonumber\\&
\leq 4\delta^{-1}\Bigg[\E(\langle \widehat{Y}_1,y\rangle_{HS}^2)\mathds{1}(|\langle \widehat{Y}_1,y\rangle_{HS}|> \varepsilon \tau_k^*/2)+\E\left(\dfrac{\sqrt{b}}{N}\sum_{i=1}^{b-1}\left(1-\dfrac{i}{b}\right)[\langle \widetilde{Z}_i,y\rangle_{HS}+\langle \widetilde{Z}_{n-i+1} ,y\rangle_{HS}]\right)^2\Bigg]\nonumber\\&
\leq 4\delta^{-1}\E(\langle \widehat{Y}_1,y\rangle_{HS}^2)\mathds{1}(|\langle \widehat{Y}_1,y\rangle_{HS}|> \varepsilon \tau_k^*/2)+O(b^3/n^2).\label{eq:voithitikiLinder}
\end{align}
By Lemma~$4$ of Kokoszka and Reimherr (2013) it follows  that $\sum_{s=-\infty}^{\infty}\E\langle Z_0,y\rangle_{HS}\langle Z_s,y\rangle_{HS}$ converges absolutely. By Kronecker's lemma, we then get, as $n\to\infty,$
\begin{align*}
\E(\langle \widehat{Y}_1,y\rangle_{HS}^2)&=\dfrac{1}{b}\sum_{i=1}^b\sum_{j=1}^b\E[\langle \widehat{Z}_i,y \rangle_{HS}\langle \widehat{Z}_j,y \rangle_{HS}]\\&=\sum_{|s|<b}\left(1-\dfrac{|s|}{b}\right)\E[\langle \widehat{Z}_0,y \rangle_{HS}\langle \widehat{Z}_s,y\rangle_{HS}]\\& =\sum_{|s|<b}\left(1-\dfrac{|s|}{b}\right)\E[\langle Z_0,y \rangle_{HS}\langle Z_s,y \rangle_{HS}]+O(b/n^{1/2}) \\&\to\sum_{s=-\infty}^{\infty}\E[\langle Z_0,y\rangle_{HS}\langle Z_s,y\rangle_{HS}].
\end{align*}
Therefore, by the dominated convergence theorem, 
\begin{equation}
\E[\langle \widehat{Y}_1,y\rangle_{HS}^2) \mathds{1}(|\langle \widehat{Y}_1,y\rangle_{HS}|> \varepsilon \tau_k^*/2)=o(1) \label{eq:dominated1}
\end{equation}
and, 
therefore, assertion (i) is proved. 

To establish assertion \ref{condition2}, notice first that 
$$
\E^*\left\|\dfrac{1}{\sqrt{k}}\sum_{t=1}^{k}\widehat{Y}^*_t\right\|_{HS}^2=\E^*\|\widehat{Y}^*_1\|_{HS}^2.
$$
Furthermore, since
\begin{align*}
\E^*\left(\dfrac{1}{\sqrt{b}}\sum_{t=1}^{b}Z_t^*\right)=\dfrac{1}{N}\sum_{t=1}^{N} \widetilde{Y}_t&=\dfrac{\sqrt{b}}{N}\left[\sum_{t=1}^{n}\widetilde{Z}_t-\sum_{i=1}^{b-1}\left(1-\dfrac{i}{b}\right) [\widetilde{Z}_i+\widetilde{Z}_{n-i+1}]\right]\\&
=\sqrt{b}\widehat{\widetilde{C}}_{n}-\dfrac{\sqrt{b}}{N}\sum_{i=1}^{b-1}\left(1-\dfrac{i}{b}\right)[\widetilde{Z}_i+\widetilde{Z}_{n-i+1}],
\end{align*}
we get
\begin{align*}
\E^*\|\widehat{Y}^*_1\|_{HS}^2&=\E^*\left\|\dfrac{1}{\sqrt{b}}\sum_{t=1}^{b}\widehat{Z}^*_t+\dfrac{\sqrt{b}}{N}\sum_{i=1}^{b-1}\left(1-\dfrac{i}{b}\right)[\widetilde{Z}_i+\widetilde{Z}_{n-i+1} ]\right\|_{HS}^2\\&=\dfrac{1}{N}\sum_{t=1}^{N}\left\|\widehat{Y}_t+\dfrac{\sqrt{b}}{N}\sum_{i=1}^{b-1}\left(1-\dfrac{i}{b}\right)[\widetilde{Z}_i+\widetilde{Z}_{n-i+1} ]\right\|_{HS}^2.
\end{align*}
Since, $\sqrt{b}N^{-1}\sum_{i=1}^{b-1}\left(1-\dfrac{i}{b}\right)[\widetilde{Z}_i+\widetilde{Z}_{n-i+1}]=O_P(b^{3/2}/n),$ it suffices to prove that the limit 
\begin{equation}
\lim_{n\to\infty}\dfrac{1}{N}\sum_{t=1}^{N}\|\widehat{Y}_t\|_{HS}^2
\label{eq:condition2i}
\end{equation}
exists and it is finite. Let $Y_t=b^{-1/2}(Z_t+\dots+Z_{t+b-1}),\,t=1,2,\ldots N,$ and note that
$
N^{-1}\sum_{t=1}^{N}\|\widehat{Y}_t\|_{HS}^2=N^{-1}\sum_{t=1}^{N}\|Y_t+\sqrt{b}(\C_0-\widehat{\widetilde{\C}}_{0})\|_{HS}^2
$. 
By Theorem~3 of Kokoszka and Reimherr (2013), in order to prove~\eqref{eq:condition2i}, it suffices to show that
\begin{equation}
\lim_{n\to\infty}\dfrac{1}{N}\sum_{t=1}^{N}\|Y_t\|_{HS}^2
\label{eq:condition2ii}
\end{equation}
 exists and it is finite. We have that
\begin{align}
\dfrac{1}{N}\sum_{t=1}^{N}\|Y_t\|_{HS}^2&= \dfrac{1}{N}\langle Z_t,Z_t\rangle_{HS}+\sum_{i=1}^{b-1}\left(1-\dfrac{i}{b}\right)\dfrac{1}{N}\sum_{t=1}^{n-i}[\langle Z_t,Z_{t+i}\rangle_{HS}+\langle Z_{t+i},Z_t\rangle_{HS}]\nonumber\\&\hspace{20pt}
-\dfrac{1}{N}\sum_{t=1}^{b-1}\left(1-\dfrac{t}{b}\right)[\langle Z_t,Z_t \rangle_{HS} +\langle X_{n-t+1},X_{n-t+1}\rangle]_{HS}\nonumber\\&\hspace{20pt}
-\dfrac{1}{N}\sum_{t=1}^{b-1}\sum_{j=1}^{b-t}\left(1-\dfrac{t+j}{b}\right)
[\langle Z_j,Z_{j+t}\rangle_{HS}+\langle Z_{n-j+1-t},Z_{n-j+1}\rangle_{HS}\nonumber\\&\hspace{145pt}+
\langle Z_{j+t},Z_{j}\rangle_{HS}+\langle Z_{n-j+1},Z_{n-j+1-t}\rangle_{HS}]\nonumber\\&
=\dfrac{1}{N}\sum_{t=1}^{n}\langle Z_t,Z_t\rangle_{HS}+ \sum_{i=1}^{b-1}\left(1-\dfrac{i}{b}\right)\dfrac{1}{N}\sum_{t=1}^{n-i}[\langle Z_t,Z_{t+i}\rangle_{HS}+\langle Z_{t+i},Z_t\rangle_{HS}]+O_P(b^2/n)\nonumber\\&= \sum_{i=-b+1}^{b-1}\left(1-\dfrac{i}{b}\right)\dfrac{1}{n}\sum_{t=1}^{n-|i|}\iint Z_t(u,v)Z_{t+|i|}(u,v)\mathrm{d}u\mathrm{d}v+O_P(b^2/n).
\end{align}
Hence, by letting $g_b(s)=\left(1-|s|/b \right)$ in Lemma~\ref{lemma:MBBtight}, we get that the last term above converges to $\sum_{s=-\infty}^{\infty}\E \iint Z_0(u,v)Z_s(u,v)\mathrm{d}u\mathrm{d}v, $ from which we conclude that,  as $n\to\infty$,
\begin{equation*}
\E^*\|Y^*_1\|_{HS}^2\to\sum_{s=-\infty}^{\infty}\E \iint Z_0(u,v)Z_s(u,v)\mathrm{d}u\mathrm{d}v,
\end{equation*}
in probability.
\end{proof}
\medskip

\noindent
\begin{proof}[Proof of Lemma~\ref{thm:katanomiTmH0}]
Using Theorem~$3$ of Kokoszka and Reimherr (2013) it follows that there exist two independent, mean zero, Gaussian Hilbert Schmidt operators $\mathcal{Z}_{1,0}$ and $\mathcal{Z}_{2,0}$ with covariance operators $\Gamma_{1,0}$ and $\Gamma_{2,0}$ respectively, such that  $$\left(\sqrt{n_1}(\widehat{\C}_{1,0}-\C_{1,0}),\sqrt{n_2}(\widehat{\C}_{2,0}-\C_{2,0})\right)$$
converges weakly to $(\mathcal{Z}_{1,0},\mathcal{Z}_{2,0}).$ Since
$$
\sqrt{\dfrac{n_1n_2}{M}}(\widehat{\C}_{1,0}-\widehat{\C}_{2,0})=
\sqrt{\dfrac{n_2}{M}}\sqrt{n_1}(\widehat{\C}_{1,0}-\widetilde{\C}_0)-\sqrt{\dfrac{n_1}{M}}\sqrt{n_2}(\widehat{\C}_{2,0}-\widetilde{\C}_0),
$$
where $\widetilde{\C}_0$ is the (under $H_0$)  common lag-zero covariance operator of the two populations, we get that, for $n_1,n_2\to\infty$ and $n_1/M\to\theta,$
$$T_M\overset{d}\to\|\mathcal{Z}_0\|^2_{HS},$$
 where $\mathcal{Z}_0=\sqrt{1-\theta}\mathcal{Z}_{1,0}-\sqrt{\theta}\mathcal{Z}_{2,0}.$
\end{proof}


\medskip
\noindent
\begin{proof}[Proof of Theorem~\ref{thm:consistentMBB}]
Using the triangle inequality and the fact that $ \sqrt{n}(\widehat{\C}_{i,0}-\C_{i,0}) \Rightarrow \mathcal{Z}_{i,0}$,  $i=1,2$, 
 it suffices to prove that $T_M^*$ converges weakly to $\|\mathcal{Z}_0\|^2_{HS}$, where $\mathcal{Z}_0=\sqrt{1-\theta}\mathcal{Z}_{1,0}-\sqrt{\theta}\mathcal{Z}_{2,0}.$ This is proved along the same lines as Lemma~\ref{thm:katanomiTmH0}  using 
of Theorem~\ref{thm:CLTmbb} and the independence of the pseudo-random elements $ \overline{\mathcal Y}_{1,n_1}^*$ and $ \overline{\mathcal Y}_{2,n_2}^*$.
\end{proof}

\section*{Acknowledgements}
The authors would like to thank the reviewers for helpful comments and suggestions.

\section*{Data Availability Statement}
The data that support the findings of this study are available on request from the corresponding author. The data are not publicly available due to privacy or ethical restrictions.

\section*{Supplementary Material}
The supplementary material contains the proofs of Lemma~\ref{lemma:sigklisidiasporasMBB} and Lemma~\ref{lemma:MBBtight}.


\end{document}